\documentclass[reqno]{amsart}
\usepackage{mathrsfs}
\usepackage{amsmath}
\usepackage{amsthm}
\usepackage{amsfonts}
\usepackage{amssymb}
\usepackage{url}
\usepackage{enumerate}
\usepackage[pdftex,bookmarks=true]{hyperref}
  \usepackage[usenames, dvipsnames]{color}
\usepackage{verbatim}

\usepackage{pdfsync}

\renewcommand\Re{{\operatorname{Re}}}

\newcommand\Vol{{\operatorname{Vol}}}

\newcommand\R{{\mathbb{R}}}

\renewcommand\P{{\mathbf{P}}}
\newcommand\E{{\mathbf{E}}}

\newcommand\Z{{\mathbb{Z}}}

\newcommand\F{{\mathbf{F}}}

\newcommand\ep{\varepsilon}
\newcommand\al{\alpha}

\newcommand\Bg{{\mathbf g}}

\newcommand\Bx{{\mathbf x}}

\newcommand\CC{{\mathcal C}}

\newcommand\eps{\varepsilon}

\usepackage{fullpage}

\theoremstyle{plain}
  \newtheorem{theorem}[subsection]{Theorem}

  \newtheorem{lemma}[subsection]{Lemma}
  \newtheorem{corollary}[subsection]{Corollary}

  \newtheorem{example}[subsection]{Example}
  
  \newtheorem{remark}[subsection]{Remark}
  
  \newtheorem{claim}[subsection]{Claim}

\theoremstyle{definition}
  \newtheorem{definition}[subsection]{Definition}

\begin{document}

\title[A note on inverse results of random walks in Abelian groups]{A note on inverse results of random walks in Abelian groups}

\author{Jake Koenig}
\address{Department of Mathematics\\ The Ohio State University \\ 231 W 18th Ave \\ Columbus, OH 43210 USA}
\email{koenig.427@osu.edu}

\author{Hoi H. Nguyen}
\address{Department of Mathematics\\ The Ohio State University \\ 231 W 18th Ave \\ Columbus, OH 43210 USA}
\email{nguyen.1261@osu.edu}

\author{Amanda Pan}
\address{Department of Mathematics\\ The Ohio State University \\ 231 W 18th Ave \\ Columbus, OH 43210 USA}
\email{pan.754@osu.edu}

\thanks{The authors are supported by NSF grant DMS-1752345.}


\begin{abstract} In this short note we give various near optimal characterizations of random walks over finite Abelian groups with large maximum discrepancy from the uniform measure. We also provide several interesting connections to existing results in the literature.
\end{abstract}

\maketitle

\section{Introduction}

Let $x_1, \dots, x_n$ be iid Bernoulli random variables taking values $\pm 1$ with probability $1/2$. Given a multiset $A$ of $n$ real numbers  $a_1, \dots, a_n$, we define \footnote{It is clear that this quantity remains the same when the $x_i$ takes value 0 or 1 with probability 1/2.}
\begin{equation}\label{eqn:classical}
\rho(A) := \sup_{a} \P\Big(\sum_{i=1}^n a_ix_i =a\Big).
\end{equation}
 In their study of roots of random polynomials in the 1940s, Littlewood and Offord \cite{LO} raised the question of bounding $\rho(A)$. They showed that if the $a_i$ are nonzero then $\rho(A)=O(n^{-1/2}\log n)$. Shortly after, Erd\H{o}s \cite{E} gave a combinatorial proof of the refinement $\rho(A)\le \binom{n}{n/2}2^{-n}$, which is the optimal bound with no further assumptions on $A$. 

Subsequently, there have been various stronger bounds by Erd\H{o}s and Moser \cite{EM}, Hal\'asz \cite{H}, Katona\cite{Kat}, Kleitman \cite{Kle}, S\'ark\"ozy and Szemer\'edi \cite{SSz}, and Stanley \cite{Stan} for different conditions on the $a_i$.  More recently, motivated by inverse questions in Additive Combinatorics, Tao and Vu initiated a new (inverse) direction to characterize $A$ for which $\rho(A)$ is large, say $\rho(A)\ge n^{-C}$ for some $C>0$ and $n\to \infty$. 

From the inverse perspective, because $A$ has $2^n$ subsums, $\rho({A})\ge n^{-C}$ means that at least $2^n n^{-C}$ among these have the same value. This suggests that  the set should have a rich ``structure". To be more precise, let us recall the notion of \emph{generalized arithmetic progressions} (GAPs). 

\begin{definition}\label{def:GAP} A subset $P$ of $\R$ is a \emph{GAP of rank $r$}, where $r\ge 1$, if it can be expressed in the form 
$$P= \Big\{g_0+ m_1g_1 + \dots +m_r g_r\Big| m_i\in\Z, N_i \le m_i \le N_i'\Big\}.$$ 
The $g_i \in \R$ are the \emph{generators} of $P$. The integer numbers $N_i,N_i'$ are the {\it dimensions} of $P$. We say that $P$ is \emph{proper} if every element of $P$ is equal to a unique such linear combination of the generators. If $N_i=-N_i'$ for all $i$ and if $g_0=0$, we say that $P$ is {\it symmetric}. 

	Finally, when $P$ is symmetric, for $t\in \Z^+$, we define 
$$P_t :=  \Big\{m_1g_1 + \dots +m_r g_r \Big | -tN_i \le m_i \le tN_i\Big\}.$$
\end{definition}

In what follows, given two sets $A,B$, their (Minkowski) sum is defined as the set 
$$A+B:= \{a+b, a\in A, b\in B\}.$$ 
We will write $2A$ for $A+A$. For instance with $P$ as in Example \ref{def:GAP} we have $nP=\{ng_0+ m_1g_1 + \dots +m_r g_r| nN_i \le m_i \le nN_i'\}$, and hence $|nP| \le \prod_{i=1}^r(nN_i' -nN_i+1)$.  Note that if $P$ is symmetric then by Definition \ref{def:GAP} we have $tP = P_t$ for any positive integer $t$.

\begin{example} Assume that $P$ is a proper symmetric GAP of rank $r=O(1)$ and cardinality $n^{O(1)}$, and that all elements of $A$ are contained in $P$. Then as $|nP|\le n^r|P|$, we have  $\rho(A) = \Omega(n^{-O(1)})$. 
\end{example}
The above example shows that if the elements of $A$ belong to a symmetric proper GAP with small rank and small cardinality, then
$\rho(A)$ is large. Tao and Vu \cite{TVstrong,TVinverse}, Vu and the second author \cite{NgV-Adv}, and more recently Tao \cite{Tao}, have justified that these are essentially the only multisets having $\rho(A)$ of polynomial growth \footnote{That is when $\rho(A)\ge n^{-O(1)}$; it is natural to ask what if $\rho\ge \exp(-n^c)$ (subexponential) or $\rho \ge \exp(-cn)$ (exponential), but we are not focusing on these regimes in this note.}. 
\begin{theorem}[Inverse Littlewood-Offord result for $\rho$]\label{theorem:ILO}
Let $\eps<1$ and $C$ be positive constants. Assume that $A=\{a_1,\dots, a_n\}$ is a multiset of real numbers
$$\rho (A)  \ge  n^{-C}. $$
Then, for any $n^\ep \le n' \le n$, there exists a proper symmetric GAP $P$ of rank $r=O_{\ep,C}(1)$ that contains all but $n'$ elements of $A$ (counting multiplicity), where 
$$|P|=\max\Big \{1, O_{C,\ep}(\rho^{-1}/{(n')}^{r/2}) \Big\}.$$ 
\end{theorem}
 
In this note, by using the simple machinery from \cite{NgV-Adv}, in combination with general John-type results for sumsets developed by Tao and Vu from \cite{TVjohn}, we will extend the above theorem to several settings of interest. Our main contributions include inverse results (1) for general finite Abelian groups (Thereom \ref{theorem:ILO:Abelian}); (2) for random walks with constraints (Theorem \ref{theorem:ILOab}); and (3) for classical random walks (Theorem \ref{theorem:ILO:m} and Theorem \ref{theorem:ILO:m:Abelian}). Additionally, we will include various interesting applications such as Corollaries \ref{cor:Abelian:1} and \ref{cor:mix}.

{\bf Notations.} We say that $X \asymp Y$ if $X=O(Y)$ and $Y=O(X)$. We say that $X=\Omega(Y)$ if $X \ge CY$ for some absolute positive constant $C$. Given a parameter $\al$, we say that $X = O_\al (Y)$, or $X \ll_\alpha Y$, if $X \le CY$ and $C$ is allowed to depend on $\alpha$. 

For any $x\in \R$, we define $\|x\|:=\|x\|_{\R/\Z}$ to be the distance of $x$ to the nearest integer. We define $e(x) = \exp(2\pi \sqrt{-1}x )$. 

Finally, if not specified otherwise, the parameter $n$ in this note is assumed to be sufficiently large.

\subsection{Inverse Littlewood-Offord in general finite Abelian groups}

Let $G$ be an additive finite Abelian group and $A=\{a_1,\dots, a_n\}$ be a multiset in $G$. Let $\xi$ be a random variable valued in $\Z$ and define 
 $$\rho_\xi(A) := \sup_{a\in G} \Big|\P\Big(\sum_{i=1}^n a_ix_i =a\Big) -\frac{1}{|G|}\Big|,$$
where  $x_i$ are iid copies of $\xi$. Hence $\rho_\xi$ measures the discrepancy of the random walk $S=\sum_{i=1}^n a_ix_i$ from the uniform measure over $G$.

Next we introduce structures to work with in the finite Abelian group setting. Generalized arithmetic progressions can be defined the same way as in Definition \ref{def:GAP} with generators $g_i$ now from $G$. Here we introduce a more general structure called a {\it coset-progression}.

\begin{definition}\cite[Chapter 5]{TVbook} A coset-progression in $G$ is a set of form $H+P$ where $P$ is a GAP and $H$ is a finite subgroup of $G$. We say that $H+P$ is \textit{symmetric} if $P$ is symmetric as a GAP, that $H+P$ has rank $r$ if $P$ has rank $r$, and that $H+P$ is \textit{proper} if $P$ is proper and $|H+P| = |H||P|$. In addition, for symmetric coset-progressions, we say that $H+P$ is $t$-\textit{proper} if the dilate $H+P_t$ is proper.
\end{definition}

Now we state our main result for this setting.
 
 \begin{theorem}[Inverse Littlewood-Offord for general Abelian groups in sparse setting, main result I]\label{theorem:ILO:Abelian} Let $\eps<1/2$ and $C$ be positive constants. Let $\al$ be a parameter that might depend on $n$, and let $\xi$ be a lazy Bernoulli r.v. of parameter $\al$ (i.e. $\P(\xi=0)=1-\al$ and $\P(\xi=\pm 1) =\al/2$). Assume that
$$\rho_\xi (A)  \ge  n^{-C}.$$
 \begin{enumerate}[(i)] 
 \item If
 $n^{\eps-1}\le \al \le 1-n^{\eps-1}$,
then for any $\alpha^{-1} n^{\ep} \le  n' \le n$, there exists a symmetric proper coset-progression $H+P$ of rank $r =O_{C,\eps}(1)$ that contains all but $n'$ elements of $A$ (counting multiplicity), where 
$$|H+P| \le  \max\Big \{1, O_{C,\eps}\Big(\frac{\rho_\xi^{-1}(A)}{(\min\{\al, 1-\al\} n')^{r/2}}\Big) \Big \}.$$ 
\item If 
$$1/2 \le \al \le 1,$$
then for any $\alpha^{-1} n^{\ep} \le  n' \le n$, there exists a symmetric proper coset-progression $H+P$ of rank $r =O_{C,\eps}(1)$ that contains all but $n'$ elements of the multiset $\{2a, a\in A\}$ (counting multiplicity), where 
$$|H+P| \le  \max\Big\{1,O_{C,\eps}\Big(\frac{\rho_\xi^{-1}(A)}{(n')^{r/2}}\Big)\Big \}.$$ 
\end{enumerate}
\end{theorem}
 Our first result is near optimal when $\al$ is not too close to the edges $0$ or $1$. It is less useful for $\al$ close to 1 (obviously things are less interesting when $\al$ is close to 0). This is natural because for instance if $G$ has characteristic 2 and $\al=1$, then the random walk is equal to $\sum_i a_i$ with probability one for every $A$, so there is no non-trivial characterization in this case. To compensate, in our second statement we allow $\al$ to be close to 1, but the structures are stated for the multiset $\{2a, a\in A\}$ rather than for $A$. 
 
 We now deduce a simple consequence, where for convenience we restrict to the non-lazy Bernoulli r.v. case.

\begin{corollary}\label{cor:Abelian:1} Let $\al=1$ and $0<\eps<1$ be a constant. Then there exists a constant $C_\eps$ such that the following holds. Let $G$ be a finite Abelian group of size at least $C_\eps\sqrt{n}$ with no proper subgroup of size less than $C_\eps \sqrt{n}$ containing all but $\eps n$ elements of $\{2a, a\in A\}$  (counting multiplicity). Then
$$ \sup_{a\in G} \P\Big(\sum_{i=1}^n a_ix_i =a\Big) =O_\eps (1/\sqrt{n}).$$
Consequently, assume that $G=\Z/q\Z$ where $q \ge 2C_\eps \sqrt{n}$ and at least $\eps n$ of the $a_i$ (counting multiplicity) are reduced modulo $q$. Then
$$ \sup_{a\in G} \P(\sum_{i=1}^n a_ix_i =a) = O_\eps (1/\sqrt{n}).$$
\end{corollary}
A proof of this can be found in Section \ref{appendix:lemma}. We note that our result for $\Z/q\Z$ above is similar to an old result of \cite[Theorem 1]{VW} by Vaughan and Wooley where it was assumed that $q$ has order at least $n$ (instead of order at least $\sqrt{n}$ as above) and all of the $a_i$ are reduced.

 \subsection{Inverse results for random walks with constraints over real numbers} Motivated by a combinatorial model of random matrices, the second author considered in \cite{Ng-Comb} the following variant of \eqref{eqn:classical}. Let $A$ be a multiset of $n$ {\it real numbers} $a_1,\dots, a_n$. Assume that $n$ is even and define
$$\rho^\ast(A):=\sup_a\P_\Bx(x_1a_1+\dots+a_{n}x_n=a),$$ 
where the probability is taken uniformly over all $0-1$ tuples $(x_1,\dots,x_n)$ with exactly $n/2$ zero entries. 

The question is that, assuming $\rho^\ast(A) \ge n^{-C}$, can we still say useful things about the $a_i$ as in Theorem \ref{theorem:ILO}? The answer is certainly yes, but can we give a {\it near optimal} characterization? 

First, observe that
\begin{equation}\label{eqn:rho:rho'}
\rho(A)=\Omega(\rho^{\ast}(A)/\sqrt{n}).
\end{equation}
 Hence, in principle we can apply Theorem \ref{theorem:ILO} to deduce some useful information on the $a_i$. We cite here results from \cite[Theorem 2.2, Theorem 2.3]{Ng-Comb}. 
\begin{theorem}\label{theorem:ILOab}
Let $\eps<1$ and $C$ be positive constants. Assume that
$$\rho^\ast(A)  \ge  n^{-C}. $$
Then, for any $n^\ep \le n' \le n$, there exists a proper symmetric GAP $P$ of rank $r=O_{\ep,C}(1)$ that contains all but $n'$ elements of $A$ (counting multiplicity), where 
$$|P|=O_{C,\ep}\Big((\rho^\ast)^{-1}\sqrt{n}\big/({n'})^{r/2}\Big).$$ 
Furthermore, assume that $n^{\ep}\le n' < n$ and $A=\{a_1,\dots,a_n\}$ is a multiset for which there are no more than $n-n'-1$ elements taking the same value. Then, there exists a (not necessarily symmetric) proper GAP $P$ of rank $2\le r=O_{\ep,C}(1)$ that contains all but $n'$ elements of $A$ (counting multiplicity), where 
$$|P|=O_{C,\ep}\Big((\rho^\ast)^{-1}\sqrt{n}\big/(n')^{r/2}\Big).$$ 
\end{theorem}

We remark that the essential advantage of the second statement over the first statement is that the rank $r$ must be at least 2, which leads to a ``gain'' of a factor $\sqrt{n'}$ in the cardinality of $|P|$. In any case, the results above are not sharp, and one of our main goals is to provide a sharper (and near optimal) result.

\begin{theorem}[Inverse Littlewood-Offord result for $\rho^\ast$]\label{theorem:ILO:ast}
Let $\eps<1$ and $C$ be positive constants. Assume that $n^{\ep}\le n' < n$ and 
$$\rho^\ast(A)  \ge  n^{-C}. $$
Then there exists a (not necessarily symmetric) proper GAP $P$ of rank $r'=O_{\ep,C}(1)$ that contains all but $n'$ elements of $A$ (counting multiplicity), where 
$$|P|=O_{C,\ep}\left(\sqrt{\frac{n}{n'}}(\rho^\ast)^{-1}\big/(n')^{r'/2}\right).$$ 
\end{theorem}

Note that this result is an improvement by a factor of $\sqrt{n}$ from the first statement of Theorem \ref{theorem:ILOab}, and that if there are no more than $n-n'-1$ elements taking the same value then $r'\ge 1$. It can also be seen that our result is near optimal by considering $a_1,\dots, a_n$ sampled randomly at uniform from a symmetric GAP of bounded rank.  

We now demonstrate two quick consequences of Theorem \ref{theorem:ILO:ast} (see Section \ref{appendix:lemma} for a proof) motivated by the classical inequality of Erd\H{o}s and Littlewood-Offord, and results of Erd\H{o}s-Moser \cite{EM} and S\'ark\"ozy-Szemer\'edi \cite{SSz} that if $a_i$ are distinct then $\rho(A)=O(n^{-3/2})$. 

\begin{corollary}\label{cor:forward}
Let $\eps<1$ be a positive constant. Assume that $n^{1/2+\ep}\le n' \le n$ and $A=\{a_1,\dots,a_n\}$ is a multiset where there are no more than $n-n'-1$ elements taking the same value. Then we have 
$$\rho^\ast(A)=O_\ep(\sqrt{n}/n').$$ 
Furthermore, if  the $a_i$ are distinct then 
$$\rho^\ast(A)=O(n^{-3/2}).$$
\end{corollary}


While Theorem \ref{theorem:ILO:ast} is useful, it is natural to study the question when the number of ones among the $x_i$ is not necessarily $n/2$. Motivated by this, for each $1\le m\le n$, we generalize $\rho^\ast$ to $\rho_{m}^{\ast}$ in which the number of ones taken in the sum is given by the parameter $m$. Explicitly,
$$\rho_{m}^{\ast}(A):=\sup_{a} \frac{\#\{(i_1,\dots, i_{m})\in \binom{[n]}{m}, \sum_{j=1}^m a_{i_j}=a \}}{\binom{n}{m}}.$$
Thus for $m= \lfloor n/2\rfloor $ we have that $\rho_{m}^{\ast}(A)= \rho^{\ast}(A)$. Now we state our main result for this more general constraint. 

\begin{theorem}[Inverse Erd\H{o}s-Littlewood-Offord for $\rho_m^\ast$, main result II]\label{theorem:ILO:ast:sparse} Let $\eps<1$ and $C$ be positive constants. Assume that $n^\eps \le m \le n-n^\eps$ and
	$$\rho_{m}^\ast(A) \ge  n^{-C}.$$
	Then for any $n^{\ep/2} (n/m) \le n' \le n$ there exists a proper GAP $Q$ of rank $r=O_{C,\ep}(1)$ which contains all but $n'$ elements of $A$ (counting multiplicity), where 
	$$|Q|\le \max\left \{1, O_{C,\ep}\left(\sqrt{n/n'}(\rho_m^\ast)^{-1}\big/(m n'/n)^{r/2}\right)\right \} .$$ 
\end{theorem}

We immediately deduce the following analog of Corollary \ref{cor:forward} (whose proof is similar to that of Corollary \ref{cor:forward}, and is also deferred to Section \ref{appendix:lemma}).

\begin{corollary}\label{cor:forward:rho_m}
	Let $\eps<1$ be a positive constant. Assume that $n^\eps \le m \le n-n^\eps$ and $n^{\ep/2} (n/m) \le n' \le n$ and $A=\{a_1,\dots,a_n\}$ is a multiset where there are no more than $n-n'-1$ elements taking the same value. Then we have 
	$$\rho_m^\ast(A)=O_\ep\Big(\frac{n}{n'\sqrt{m}}\Big).$$ 
	In particular, 
	\begin{itemize}
	\item if there are no more than $(1-\eps)n$ elements of $A$ taking the same value, then
	$$\rho_m^\ast(A)=O_\ep\Big(\frac{1}{\sqrt{m}}\Big);$$ 
	\item if  the $a_i$ are distinct, then 
	$$\rho_m^\ast(A)=O_\eps\Big(\frac{1}{n\sqrt{m}}\Big).$$
	\end{itemize}
\end{corollary}

\subsection{Inverse results for classical random walks} After obtaining a somewhat near optimal characterization for $\rho_m^\ast$, we connect our result to classical random walks. We can view the concentration of $\rho_m^\ast$ as choosing a random word $(a_{i_1},\dots, a_{i_m})$ uniformly at random, where $i_j \neq i_k$ for $j\neq k$. It thus makes sense to study random words of length $m$ where $i_j$ is not necessarily different from $i_k$. We define
$$\rho_{m}(A):=\sup_{a} \frac{\#\{(i_1,\dots, i_{m})\in [n]^m, \sum_j a_{i_j}=a \}}{n^m}.$$
In other words, we can view $\rho_m(A)$ as 
$$\rho_m(A) = \sup_{a} \P(S=a),$$
where $S=X_1+\dots+X_m$ and $X_i$ are iid uniform in $\{a_1,\dots, a_n\}$. This is the concentration probability for a random walk of length $m$ where the possible steps are drawn from the set $A$. Now we give our result for this model, stated in the symmetric setting. \footnote{It might be possible to extend our result to the non-symmetric setting, although we are not pursuing it here.}
\begin{theorem}[Inverse result for $\rho_m$ in torsion-free setting]\label{theorem:ILO:m} Let $\eps<1$ and $C$ be positive constants. Assume that $A$ is a symmetric set of real numbers (i.e. $a\in A$ implies $-a\in A$) and that
 $$\rho_{m}(A) \ge  m^{-C}$$
where $m$ is sufficiently large. Then for any $\eps n \le n' \le n$ there exists a symmetric proper GAP $Q$ of rank $r=O_{C,\ep}(1)$ which contains all but $n'$ elements of $A$ (counting multiplicity), where 
$$|Q|\le \max\left \{1, O_{C,\ep}\left( (\rho_m)^{-1}/(m n'/n)^{r/2} \right) \right \} .$$ 
\end{theorem}

First, we note that in Theorem \ref{theorem:ILO:m} there is no connection between $m$ (the number of steps) and $n$ (the number of elements of the multiset $A$). We can also open the range of $n'$ to $n^\eps \le n' \le n$ but here we restrict to $\eps n \le n' \le n$ for simplicity.

Next, we remark that as $A$ is a multiset, the elements of $A$ may be repeated. For instance $A$ can have the form $\{-a_1^{[s_1]},a_1^{[s_1]} \dots, -a_k^{[s_k]}, a_k^{[s_k]}\}$ (where $\sum_i 2s_i =n$), in which case $X_i$ are random variables taking values $a_i$ with probability $s_i/n$. Especially in the case $k$ is fixed and $s_i \approx n/k$, by choosing $n' =(1-1/k)n$ for instance, then by our result $\rho_m(A) =O( \min_{Q} 1/(|Q|m^{r/2}))$ where the minimum is taken over all GAP $Q$ of rank $r$ that contains all $a_1,\dots, a_k$. 


Finally, it is natural to consider $\rho_m$ for finite Abelian groups. Here we state Theorem \ref{theorem:ILO:m} in the following form (see Remark \ref{rmk:quantitative}).

\begin{theorem}[Inverse result for $\rho_m$ in finite Abelian setting, main result III]\label{theorem:ILO:m:Abelian} Let $G$ be a finite Abelian group. Let $\eps<1$ and $C$ be positive constants. Assume that $A\subset G$ is symmetric, and that
 $$\rho_{m}(A):=\sup_{a\in G} \Big|\frac{\#\{(i_1,\dots, i_{m})\in [n]^m, \sum_j a_{i_j}=a \}}{n^m}-\frac{1}{|G|}\Big| \ge  m^{-C}$$
where $m$ is sufficiently large. Then for any $\eps n \le n' \le n$ there exists a symmetric proper coset-progression $Q$ of rank $r=O_{C}(1)$ which contains all but $n'$ elements of $A$ (counting multiplicity), where 
$$|H+Q|\le \max\left \{1, O_{C,\ep}\left( (\rho_m)^{-1}/(m n'/n)^{r/2} \right) \right \} .$$ 
\end{theorem}

In what follows we deduce an interesting corollary. As in the discussion preceding Theorem \ref{theorem:ILO:m:Abelian}, consider $A\subset G$ of the form
\begin{equation}\label{eqn:rho_m:A}
A=\{-a_1^{[s_1]},a_1^{[s_1]} \dots, -a_k^{[s_k]}, a_k^{[s_k]}\}, \mbox{ where $a_i \neq a_j$ and $\sum_i 2s_i =n$.}
\end{equation} 
Recall that in this case, for each $a\in G$,  
$$ \P(X_1+\dots+X_m=a) = \frac{\#\{(i_1,\dots, i_{m})\in [n]^m, \sum_j a_{i_j}=a \}}{n^m}$$ 
where $X_1,\dots, X_m$ are iid random variables taking values $\pm a_i$ with probability $p_i=s_i/n$ (and $A$ from \eqref{eqn:rho_m:A} is expressed in the form $\{a_1, \dots, a_n\}$).

\begin{corollary}[Random walks over random symmetric generating sets]\label{cor:mix} Let $0<\eps<1$ and $k \in \Z^+$ be fixed. There exists a constant $C=C({\eps, k})$ such that the following holds. Let $\delta$ be a parameter that might depend on $q$. Assume that $A$ is as in \eqref{eqn:rho_m:A} with $\eps \le p_1,\dots, p_k \le 1-\eps$. Then if there is no symmetric proper coset-progression $H+Q$ of rank $r$ for some $r\le k-1$ and size $C^{-1}  \delta^{-1}|G|/m^{r/2}$ that contains all $a_1,\dots, a_k$. Then for $m \ge C \delta^{-2/k}  |G|^{2/k}$ we have 
$$\left|\P(X_1+\dots+X_m=a) -\frac{1}{ |G|}\right| \le \frac{\delta}{ |G|}.$$
In particular, assuming that $q$ is sufficiently large and $a_1,\dots, a_k$ are chosen uniformly from the set of reduced elements of $G=\Z/q\Z$, then for $A$ as in \eqref{eqn:rho_m:A} with  $\eps \le p_1,\dots, p_k \le 1-\eps$ and any $t$ (that might depend on $q$), with probability at least $1-O(t^{-k})$ the random walk $X_1+\dots+X_m$ is $\delta$-mixing \footnote{Which means that for all $A\subset \Z/q\Z$, $|\P(X_1+\dots+X_m \in A) - |A|/q|\le \delta$.} provided that
$$m \ge C t \delta^{-2} (\log \log q)^2 (q \log q)^{2/k}.$$ 
\end{corollary}
We will present a proof of Corollary \ref{cor:mix} in Section \ref{appendix:lemma}. Regarding the second application, it might be possible to replace $\Z/q\Z$ by other finite Abelian groups with relatively few subgroups, however we will not focus on this aspect here. For $\Z/q\Z$, one sees that the bound $q^{2/k}$ is necessary. Indeed, heuristically, the random walk over $\{-a_1,a_1,\dots, a_k,-a_k\}$ after $m$ steps will concentrate mostly on the GAP $\{\sum_{i=1}^k x_i a_i, |x_i| = O(\sqrt{m})\}$. The volume of this GAP is  bounded by $(C \sqrt{m})^k$, hence to expect that the random walk uniformly covers $\Z/q\Z$, one must have $m \gg q^{2/k}$. 

We also refer the reader to several results by Hildebrand and coauthors in \cite{DH, HB,HB'} for similar statements with non-symmetric random walks in $\Z/q\Z$. More concretely, \cite[Theorem 1]{DH} says that the $m$-step random walk over $\{a_1,\dots, a_k\}$, where $a_1,\dots, a_k$ are chosen uniformly from the set of $k$ tuples where $\{a_i=a_j, 1\le i\neq j \le k\}$, generates $\Z/q\Z$, has the property that $\E (\sum_a |\P(X_1+\dots+X_m=a) -\frac{1}{ |G|}|) \to 0$ as long as $m \ge C_q q^{2/(k-1)}$ and $C_q \to \infty$ with $q$. (Here the bound $q^{2/(k-1)}$ is necessary because heuristically, the $m$-step random walk $X_1+\dots+ X_m$ can be written as $ma_1 + Y_1+\dots+ Y_m$, where $Y_i$ are iid uniform over $\{0,a_2-a_1,\dots, a_k-a_1\}$. The random walk $Y_1+\dots+Y_m$ concentrates mostly on the GAP $ma_1+ \{\sum_{i=2}^{k} x_i (a_i-a_1), |x_i| = O(\sqrt{m})\}$, whose volume is bounded by $(C \sqrt{m})^{k-1}$. Hence to expect near uniform distribution one must have $m \gg q^{2/(k-1)}$.) 


We complete our introduction with two remarks.

\begin{remark}\label{rmk:Tao} It seems possible to deduce some variant of Theorem \ref{theorem:ILO:Abelian} from \cite[Section 7]{Tao}, and similarly, some variant of Theorem \ref{theorem:ILO:m:Abelian} from \cite[Theorem 1.12]{Tao}, especially when $|G|$ is larger than $n^{O(1)}$. The setting of \cite{Tao} applied to Abelian groups, however, is different in that there one works with the $\|.\|_2$-concentration (of the probability distribution $\mu$ of $S=X_1+\dots+X_m$ or $S=a_1x_1+\dots+a_n x_n$) rather than with the $\|.\|_\infty$-concentration (or more precisely, with the maximum discrepancy $\sup_a |\mu(a)-\frac{1}{|G|}|$) as in our current setting. Strictly speaking, conditions such as (1.5) of \cite[Theorem 1.12]{Tao} are automatically satisfied for any $\mu$ when $|G| =n^{O(1)}$, so in this case of $G$ one seems to need some modifications. In any case, here we hope to provide more direct proofs with explicit bounds using the elementary approach of \cite{NgV-Adv}. 
 \end{remark}


\begin{remark}\label{rmk:quantitative} Our results can be made quantitative in the following ways, where $C_0'$ is an absolute constant.

\begin{itemize}
\item In the statements of Theorem \ref{theorem:ILO:Abelian} (and its corollaries) we can allow $C$ to vary with $n$ as long as
 $$C=o(\log \log \log n).$$
In this regime, our structures (GAPs, coset-progressions) will have rank $r$ with $r\le 4C/\eps$ and cardinality bounded by
\begin{equation}\label{eqn:Ab:HP:quant}
2^{2^{2^{28C/\eps+C_0'}}} \frac{(\rho_{\xi}(A))^{-1}}{(\max\{\al, 1-\al \}n')^{r/2}}.
\end{equation}
\vskip .05in
\item In the statement of Theorem \ref{theorem:ILO:ast:sparse} we can allow $C$ to vary with $n$ as long as 
$$C=o(\log \log \log n).$$
In this regime, our structures (GAPs) will have rank $r$ with $r\le 2C/\eps$ and cardinality bounded by
\begin{align}\label{eqn:ILO:ast:sparse:quant}
2^{2^{2^{14C/\eps+C_0'}}} \sqrt{\frac{n}{n'}}  (\rho^\ast_m(A))^{-1}/(\alpha n')^{r'/2}.
\end{align}
\vskip .05in
\item In the statement of Theorem \ref{theorem:ILO:m:Abelian} we can allow $C$ to vary with $m$ as long as 
$$C=o(\log \log \log m).$$
In this regime, our structures (GAPs, coset-progressions) will have rank $r$ with $r\le 4C$ and cardinality bounded by
\begin{equation}\label{eqn:Ab:m:HP:quant}
2^{2^{2^{28C+C_0'}}} \frac{(\rho_m(A))^{-1}}{( mn'/n)^{r}}.
\end{equation}

\end{itemize}
\end{remark}


\section{Supporting lemmas}\label{section:supporting}

We will make use of two beautiful results from \cite{TVjohn} by Tao and Vu. The first result allows one to pass from coset-progressions to proper coset-progressions in an ambient Abelian group without a substantial loss.

\begin{theorem}\cite[Corollary 1.18]{TVjohn}\label{thm:proper} There exists a positive integer $C_1$ such that the following holds. Let $Q$ be a symmetric coset-progression of rank $d\ge 0$ and let $t\ge 1$ be an integer. Then there exists a $t$-proper symmetric coset-progression $P$ of rank at most $d$ such that we have
$$Q \subset P \subset Q_{{(C_1d)}^{3d/2}t}.$$ 
We also have the size bound
$$|Q| \le |P| \le t^d {(C_1d)}^{3d^2/2} |Q|.$$
\end{theorem}

The second result says that as long as $|kX|$ grows slowly compared to $|X|$, then it can be contained in a coset-progression. This is a long-ranged version of the Freiman-Ruzsa theorem.

\begin{theorem}\label{thm:longrange}\cite[Theorem 1.21]{TVjohn} There exists a positive integer $C_2$ such that the following statement holds: whenever $d,k\ge 1$ and $X \subset G$ is a non-empty finite set such that 
$$k^d|X| \ge 2^{2^{C_2 d^2 2^{6d}}} |kX|,$$
then there exists a proper symmetric coset-progression $H+Q$ of rank $0\le d'\le d-1$ and size $|H+Q| \ge 2^{-2^{C_2 d^2 2^{6d}}} k^{d'}|X|$ and $x,x' \in G$ such that 
$$x + (H+Q) \subset kX \subset x' + 2^{2^{C_2 d^2 2^{6d}}} (H+Q).$$

\end{theorem}

Note that any GAP $ Q=\{a_0+ x_1a_1 + \dots +x_r a_r|  -N_i \le x_i \le N_i \hbox{ for all } 1 \leq i \leq r\}$ is contained in a symmetric GAP  $ Q'= \{x_0 a_0+ x_1a_1 + \dots +x_r a_r| -1\le x_0 \le 1,  -N_i \le x_i \le N_i \hbox{ for all } 1 \leq i \leq r\}$. Thus, by combining Theorem \ref{thm:longrange} with Theorem \ref{thm:proper} we obtain the following

\begin{corollary}\label{cor:longrange} Whenever $d,k\ge 1$ and $X \subset G$ is a non-empty finite set such that 
$$k^d|X| \ge 2^{2^{C_2 d^2 2^{6d}}} |kX|,$$
then there exists a 2-proper symmetric coset-progression $H+P$ of rank $0\le d'\le d$ and size 
$$|H+P| \le 2^d (C_1 d)^{3d^2/2} 2^{d2^{C_2 d^2 2^{6d}}} |kX|$$ 
such that 
$$ kX \subset H+P.$$
\end{corollary}

It is desirable to improve the bounds on $|H+P|$ above. One might try to use a near-optimal version of Freiman-Ruzsa's inverse theorem from \cite{Sand} (instead of \cite{GR}) in the proofs of \cite{TVjohn}, however this does not seem to give a significant improvement.

In application we want to deduce information about $X$. For that we first need the following result.

\begin{lemma}\label{lemma:dividing} 
Assume that $0 \in X$ and that $H+P$ is a symmetric 2-proper coset-progression of the form $H+\{\sum_{i=1}^d
x_ia_i: |x_i|\le N_i\}$ that contains $kX$. Then $X\subset H + \{\sum_{i=1}^d x_ia_i: |x_i|\le
2N_i/k\}$.
\end{lemma}

\begin{proof}(of Lemma \ref{lemma:dividing})
Without a loss of generality, we can assume that $k=2^l$. It is sufficient to 
show that 
$$2^{l-1}X \subset H+\big\{\sum_{i=1}^d x_ia_i: |x_i|\le
N_i/2\big\}.$$ 
For this, because $0 \in X$, $2^{l-1}X \subset 2^lX\subset H+P$, any element
$x$ of $2^{l-1}X$ can be written as $x=h+\sum_{i=1}^d x_ia_i$, with
$|x_i|\le N_i$. Now, because $x\in 2^{l-1}X$, we have 
$$2x =2h + \sum_{i=1}^d (2x_i)a_i  \in 2^l X \subset H+P.$$
So there exist $h'\in H$ and integers $y_1,\dots,y_d$ with $|y_i|\le N_i$ such that $2h + \sum_{i=1}^d (2x_i)a_i = h'+ \sum_{i=1}^d y_i a_i$.  On the other hand, as $H+2P$ is proper (as $H+P$ is 2-proper) and the above elements are in $H+2P$, we must 
have $2h=h'$ and $2x_i = y_i$, and hence $|2x_i|\le N_i$.
\end{proof}
We note that in Lemma \ref{lemma:dividing} above, if $2N_i/k<1$ then $x_i=0$, which means that the direction $a_i$ plays no role in the information for $X$. So if we let $I$ be the collection of $i$ for which $2N_i/k \ge 1$, then $r=|I|\le d$ and
$X\subset H + \{\sum_{i\in I} x_ia_i: |x_i|\le
2N_i/k\}$, which is a symmetric coset-progression of size at most 
\begin{equation}\label{eqn:HP}
|H| \prod_{i\in I} \frac{4N_i+1}{k} \le \frac{2^r}{k^r}|H| \prod_{i=1}^d (2N_i+1) = \frac{2^r}{k^r} |H+P|.
\end{equation}

We can then deduce the following long range inverse theorem, which is an Abelian analog from \cite{NgV-Adv}.

\begin{theorem}\label{longrangeinversetheorem:Abelian}(Long Range Inverse Theorem) There exists a constant $C_0$ such that the following holds. Let $d$ be a positive integer. Assume that $X$ is a subset of a finite group $G$ such that $0\in X$ and 
$$|kX| \leq \frac{k^d}{2^{2^{2^{7d+C_0}}}}|X|,$$ 
for some integer $k\geq 2$ that may depend on $|X|$. Then there is a proper symmetric coset-progression $H+Q$ of rank $r\le d$ and cardinality 
$$|H+Q| \le  \frac{2^{2^{2^{7d+C_0}}} }{k^r} |kX|$$
such that $X\subset H+Q$.
\end{theorem}

\begin{proof}(of Theorem \ref{longrangeinversetheorem:Abelian}) Under the assumption of the theorem, as $C_0$ is sufficiently large we can apply Corollary \ref{cor:longrange} to a 2-proper symmetric coset-progression $H+P$ of rank $0\le d'\le d$ and size 
$$|H+P| \le 2^d (C_1 d)^{3d^2/2} 2^{d2^{C_2 d^2 2^{6d}}} |kX|$$ 
such that 
$$ kX \subset H+P.$$
We then apply Lemma \ref{lemma:dividing} to obtain a symmetric proper coset-progression $H+Q$ of rank $r\le d$, which contains $X$ and by \eqref{eqn:HP}
$$|H+Q| \le \frac{2^{2^{2^{7d+C_0}}} }{k^r} |kX|.$$
\end{proof}

We complete the section with a torsion-free variant from \cite{NgV-Adv}.
\begin{theorem}\label{longrangeinversetheorem} There exists a constant $C_0$ such that the following holds. Let $d$ be a positive integer. Assume that $X$ is a subset of a torsion-free group such that $0\in X$ and 
$$|kX| \leq \frac{k^d}{2^{2^{2^{7d+C_0}}}}|X|,$$ 
for some integer $k\geq 2$ that may depend on $|X|$. Then there is a proper symmetric GAP $Q$ of rank $r\le d$ and cardinality 
$$|Q| \le  \frac{2^{2^{2^{7d+C_0}}} }{k^r} |kX|$$
such that $X\subset Q$.
\end{theorem}

\section{Inverse Littlewood-Offord in general finite Abelian groups: proof of Theorem \ref{theorem:ILO:Abelian}}\label{section:Abelian}


In the following we focus on the first part of the theorem, assuming $n^{\eps-1}\le \al \le 1-n^{\eps-1}$. Modifications for the second part will be discussed at the end. 

By fixing a non-degenerate bilinear form over $G$, for any $a\in G$ we have the standard identity
\begin{align*}
	\P(S = a) = \E\frac{1}{|G|}\sum_{\zeta\in G} e(\zeta\cdot(S - a)) =\frac{1}{|G|}  \E\sum_{\zeta\in G} e(\zeta \cdot S)e(-\zeta\cdot a),
\end{align*}
where $S=\sum_{i=1}^n x_i a_i$ and $x_i$ are iid copies of $\xi$. 

Hence, assume that $a$ is what maximizes the discrepancy, so that
$$\rho_\xi(A) =\left|\P(S = a) - \frac{1}{|G|}\right| = \frac{1}{|G|} \Big| \E \sum_{\zeta\in G, \zeta \neq 0} e(\zeta \cdot S)e(-\zeta\cdot a)\Big| \le \frac{1}{|G|} \Big| \E \sum_{\zeta\in G, \zeta \neq 0} e(\zeta \cdot S)\Big|.$$
By independence,
\begin{align*}
	|\E e( \zeta \cdot S)| = \Big|\prod_{i=1}^n \E e(x_i \zeta \cdot a_i)\Big| \leq \prod_{i=1}^n |1-\alpha + \alpha\cos(2\pi \zeta \cdot a_i )|.
\end{align*}
Note that  $|\sin  \pi x | \ge 2 \|x\|$ for any $x\in \R$, where we recall that $\|x\|=\|x\|_{\R/\Z}$ is the distance of $x$ to the nearest integer. Hence
$$1-\alpha + \alpha\cos(2\pi x) = 1-\al(1-\cos(2\pi x)) = 1- 2 \al \sin^2(\pi x) \le 1 -8 \al \|x\|^2 \le \exp(-8 \al \|x\|^2)$$ 
as well as
$$-(1-\alpha + \alpha\cos(2\pi x) ) = 2 \al \sin^2(\pi x) -1 \le 1 - 2(1-\al) \sin^2 (\pi x) \le 1 -8 (1-\al) \|x\|^2 \le \exp(-8 (1-\al) \|x\|^2).$$
It thus follows that
\begin{equation} \label{eqn:fourier3-1}
|1-\alpha + \alpha\cos(2\pi x)| \le \exp(-8 \min\{\al, 1-\al\} \|x\|^2).
\end{equation}


Without loss of generality we assume that $\al\le 1/2$, and hence we obtain the following inequality,
\begin{align}\label{concentration_upper_bound:groups}
	\rho_\xi(A) \leq \frac{1}{|G|}\sum_{\zeta \neq 0}\exp(-8\alpha \sum_{i=1}^n \big\| a_i \cdot \zeta\big\|^2).
\end{align}
Combining with our assumption on $\rho_\xi(A)$ yields
\begin{align*}
	n^{-C} \leq \frac{1}{|G|}\sum_{\zeta \neq 0}\exp(-8\alpha\sum_{i=1}^n \big\| a_i \cdot \zeta\big\|^2).
\end{align*}

\emph{Large level sets.} Now we split up the $\zeta$ by their effect on the sum. Let $S_\ell = \{\zeta\ \big| 4\alpha \sum_{i=1}^n \| a_i\cdot \zeta\|^2 \leq \ell\}$. We have
\begin{align*}
	n^{-C}\leq\rho_\xi(A)\leq\frac{1}{|G|}\sum_{\zeta \neq 0}\exp(-8\alpha\sum_{i=1}^n \big\| 2\pi a_i \cdot \zeta\big\|^2) \leq  \frac{1}{|G|}\sum_{\ell\geq 1} \exp(-2(\ell-1))|S_\ell|.
\end{align*}

Because $\sum_{m\geq 1}\exp(-m)<1$, there must be a level set $S_{\ell_0}$ such that 
\begin{align}\label{eqn:Sell_0}
	|S_{\ell_0}|\exp(-{\ell_0}+2) \geq \rho_\xi |G|.	
\end{align}

Because $\rho_\xi\geq n^{-C}$ and of course $|S_{\ell_0}|\leq |G|$ we have $\ell_0 \le C \log n$.

\emph{Double counting and the triangle inequality.} By double counting we have
\begin{align*}
	\sum_{i=1}^n 4\alpha \sum_{\zeta\in S_{\ell_0}}  \| a_i \cdot \zeta\|^2 = \sum_{\zeta\in S_{\ell_0}} 4\alpha \sum_{i=1}^n  \| a_i \cdot \zeta\|^2\leq {\ell_0}|S_{\ell_0}|. 	
\end{align*}
So by averaging, at least $n-n'$ of the  $a_i$ satisfy
\begin{align}\label{averaging:groups}
	\sum_{\zeta\in S_{\ell_0}} \| a_i \cdot \zeta\|^2 \leq \frac{{\ell_0}}{4 \alpha n'}|S_{\ell_0}|.
\end{align}
Let the set of $a_i$ satisfying \eqref{averaging:groups} be $A'$. The set $A\setminus A'$ will be our exceptional set. It remains to show that $A'$ is contained in a symmetric proper coset-progression.

Let $k$ be any positive integer. By the triangle inequality we have for any $a\in kA'$, writing $a = a_1 +\dots + a_k$ with $a_i\in A'$,
\begin{align*}
	\Big\|a \cdot \zeta\Big\|^2\leq \Big(\|a_1\cdot \zeta\| + \dots + \|a_k \cdot \zeta\|\Big)^2.
\end{align*}
By Cauchy-Schwarz 
\begin{align*}
	\Big(\|a_1 \cdot \zeta\| + \dots + \|a_k \cdot \zeta\|\Big)^2 \leq k\sum_{i=1}^k \| a_i \cdot \zeta\|^2.
\end{align*}
Therefore by equation~\eqref{averaging},
\begin{align*}
	\sum_{\zeta\in S_{\ell_0}} \Big\|a \cdot \zeta\Big\|^2 \leq \sum_{\zeta\in S_{\ell_0}} k\sum_{i=1}^k \| a_i \cdot \zeta\|^2 \leq k^2 \frac{{\ell_0}}{4\alpha n'}|S_{\ell_0}|.
\end{align*}

Of course for any $a \in jA'$ where $j\leq k$, we also have
\begin{align}\label{longrangeinverse:groups}
	\sum_{\zeta\in S_{\ell_0}}\Big\|2\pi a \cdot \zeta\Big\|^2\leq k^2 \frac{{\ell_0}}{4\alpha n'}|S_{\ell_0}|.
\end{align}
\emph{Dual sets.} Define 
$$S_{\ell_0}^* := \Big \{a\ \Big|\sum_{\zeta\in S_{\ell_0}} \| a\cdot \zeta\|^2 \leq \frac{1}{200}|S_{\ell_0}| \Big\}.$$ 
This is related to the concept of a dual in that it is the set of $a$ which are nearly orthogonal to $\zeta\in S_{\ell_0}$. It gives us the following inequality which is reminiscent of the equality on cardinalities one obtains for a vector space's dual,
\begin{align}\label{bound for dual set:groups}
	|S_{\ell_0}^*|\leq \frac{4|G|}{|S_{\ell_0}|}.
\end{align}

To see this, define $T_a = \sum_{\zeta\in S_{\ell_0}} \cos (2\pi a\cdot \zeta)$. Using the fact that $\cos (2\pi z) \geq 1 - 100\|z\|^2$  for any $z\in \R$  we have for any $a\in S_{\ell_0}^*$
\begin{align*}
	T_a\geq \sum_{\zeta\in S_{\ell_0}} \left(1-100\|a \cdot \zeta\|^2\right)\geq \frac{1}{2}|S_{\ell_0}|.
\end{align*}

We also have the upper bound on $T_a$ given by,
\begin{align*}
	\sum_{a\in G} T_a^2 &= \sum_{a\in G} \Big(\sum_{\zeta\in S_{\ell_0}} \cos(2\pi a\cdot \zeta)  \Big)^2 \\
	&=\sum_{a\in G} \Big(\sum_{\zeta\in S_{\ell_0}} \Re(e(a\cdot \zeta)) \Big)^2 \\
	&\leq \sum_{a\in G} \Big(\sum_{\zeta\in S_{\ell_0}} e(a\cdot \zeta) \Big)\overline{\Big(\sum_{\zeta\in S_{\ell_0}} e(a\cdot \zeta) \Big)} \\
	&= \sum_{a\in G} \sum_{\zeta_1,\zeta_2\in S_{\ell_0}} e( a \cdot (\zeta_1-\zeta_2)) \\
	&= |G| |S_{\ell_0}|.
\end{align*}

Noting in the last equality that there are $|S_{\ell_0}|$ choices of $\zeta_1=\zeta_2$ which contribute a $|G|$, and nothing else contributes anything, (by the identity we used at the beginning that $\sum_{a\in G}\cos(2\pi a \cdot  x) = |G|\mathbb{I}_{x=0}$). Now we can average and conclude that no more than $\frac{4|G|}{|S_{\ell_0}|}$ elements $a\in G$ can be in $|S_{\ell_0}^*|$, confirming \eqref{bound for dual set:groups}. 

Set 
\begin{equation}\label{eqn:Ab:k}
k :=\left\lfloor \sqrt{\frac{{\alpha}n' }{200{\ell_0}}} \right\rfloor.
\end{equation}
Note that 
$$k\ge n^{\eps/2}/\sqrt{200 \ell_0} \ge n^{\eps/3}$$ 
as $n$ is sufficiently large.

By \eqref{longrangeinverse:groups} we have  that
$$\bigcup_{l=1}^{k} lA' \subset S_{\ell_0}^*.$$ 
Setting $A'' = A'\cup \{0\}$ we have (choosing $0$ in a sum $k-l$ times is equivalent to simply adding $l$ elements of $A'$)
$$kA'' = \{0\}  \cup \bigcup_{l=1}^{k} lA' \subset S_{\ell_0}^*.$$
This gives us the bound
\begin{align*}
	|kA''| \leq |S_{\ell_0}^*| \leq \frac{4|G|}{|S_{\ell_0}|} \leq  4\rho_\xi^{-1}\exp(-{\ell_0}+2),
\end{align*}
where in the second inequality we used \eqref{eqn:Sell_0} and in the third inequality we used \eqref{bound for dual set:groups}.

\emph{Long Range Inverse Theorem.} Recall our hypothesis that $\rho_\xi(A)\geq n^{-C}$. Therefore, with $k$ from \eqref{eqn:Ab:k} we clearly have 
$$|kA''|\leq 4 \exp(-{\ell_0}+2) n^C \le  \frac{k^{4C/\eps}}{2^{2^{2^{28C/\eps+C_0}}}} \le  \frac{k^{4C/\eps}}{2^{2^{2^{28C/\eps+C_0}}}} |A''|$$
 if we assume that $C =o( \log \log \log n)$. 
 
 It thus follows from Theorem~\ref{longrangeinversetheorem:Abelian} that $A''$ is contained in a symmetric proper coset-progression $H+Q$ of rank $r\le 4C/\eps$ and size 
\begin{align*}
|H+Q| & \le 2^{2^{2^{28C/\eps+C_0}}} 4 \exp(-\ell_0+2) \frac{(\rho_\xi(A))^{-1}}{{k}^{r}} \\
& \le 2^{2^{2^{28C/\eps+C_0}}} 4 \exp(-\ell_0+2) (\sqrt{200 \ell_0})^{r}   \frac{(\rho_\xi(A))^{-1}}{(\al n')^{r/2}}\\
& \le  2^{2^{2^{28C/\eps+C_0'}}} \frac{(\rho_\xi(A))^{-1}}{(\al n')^{r/2}},
\end{align*}
concluding the proof of the first statement of Theorem \ref{theorem:ILO:Abelian} in its quantitative form of \eqref{eqn:Ab:HP:quant}. $\hfill \qed$

Now we discuss the modifications to prove the second part of Theorem \ref{theorem:ILO:Abelian} when $1/2 \le \al \le 1$. In this case we have
$$\rho_\xi(A) =\left|\P(S = a) - \frac{1}{|G|}\right| \leq \frac{1}{|G|}\sum_{\zeta \neq 0} \prod_{i=1}^n |1-\al +\al \cos(2\pi \zeta \cdot a_i )| =\frac{1}{|G|}\sum_{\zeta \neq 0} \prod_{i=1}^n (1-\al +\al \cos(2\pi \zeta \cdot a_i )).$$
Next, as $\cos(2\pi x) \le \frac{1}{2}  + \frac{1}{2} \cos^2(2\pi x) = 1- \frac{1}{2}\sin^2(2\pi x) \le 1 - 2\|2x\|^2$  we obtain that 
$$\rho_\xi(A)  \leq \frac{1}{|G|}\sum_{\zeta \neq 0} \exp(-2\al \sum_i \|\zeta \cdot 2a_i \|^2).$$
The rest of the proof is similar to that of the first case above applied to the set $\{2a, a\in A\}$ in place of $A$; we omit the details.

\section{Inverse result for classical random walks: proof of Theorem \ref{theorem:ILO:m} and Theorem \ref{theorem:ILO:m:Abelian}} 

It suffices to prove Theorem \ref{theorem:ILO:m:Abelian}. We write $S=X_1+\dots+X_m$ where $X_i$ are chosen uniformly from $A=\{-a_1,a_1,\dots, -a_n,a_n\}$ and without loss of generality $a_i \neq 0$ for all $i$. 

We have
\begin{equation}\label{eqn:fourier1:m} \rho_m= \left|\P(S=a)-\frac{1}{|G|}\right|= \Big|\E \frac{1}{|G|} \sum_{\zeta \neq 0} e (\zeta \cdot (S-a)) \Big|=\Big| \E \frac{1}{|G|} \sum_{\zeta \neq 0} e ( \zeta \cdot   S) e(-\zeta \cdot a)\Big|  \le \frac{1}{|G|} \sum_{ \zeta\neq 0} |\E e (\zeta \cdot S)|.
\end{equation}
By independence
\begin{equation*} \label{eqn:fourier2:M}  |\E e(\zeta \cdot S)| = \prod_{j=1}^m |\E e(\zeta \cdot X_j)| = \prod_{j=1}^m \Big|\frac{1}{n} \sum^n_{i=1} \frac{1}{2}(e(\zeta \cdot a_i)+e(-\zeta \cdot a_i))\Big| =  \prod_{j=1}^m \Big|\frac{1}{n} \sum^n_{i=1} \cos(2\pi \zeta \cdot a_i)\Big|,  \end{equation*}

because of the symmetry of $A$. 

Using again $|\sin(\pi x)| \ge 2\|x\|$,
$$\frac{1}{n} \sum^n_{i=1} \cos(2\pi \zeta \cdot a_i) = 1 -\frac{2}{n} \sum_i \sin^2(\pi \zeta \cdot a_i) \le 1  -\frac{8}{n} \sum_i \|\zeta \cdot a_i\|^2 \le \exp(- \frac{8}{n} \sum_i \|\zeta \cdot a_i\|^2)$$
and
$$-\frac{1}{n} \sum^n_{i=1} \cos(2\pi \zeta \cdot a_i) =\frac{1}{n} \sum^n_{i=1} \cos(2\pi (\zeta \cdot a_i +1/2)) \le \exp(-\frac{8}{n} \sum_i \|\zeta \cdot a_i +1/2\|^2)).$$
Hence we have
$$\left|\frac{1}{n} \sum^n_{i=1} \cos(2\pi \zeta \cdot a_i)\right| \le  \exp\left(-\frac{8}{n} \min\big\{ \sum_i \|\zeta \cdot a_i\|^2, \sum_i \|\zeta \cdot a_i +1/2\|^2\big\}\right)$$

Consequently, we obtain a key inequality, where $\al= m/n$
\begin{equation} \label{eqn:fourier4:m'}
\rho_m   \le  \frac{1}{|G|} \sum_{\zeta \neq 0}  \exp\left(-8\al \min\big\{ \sum_i \|\zeta \cdot a_i\|^2, \sum_i \|\zeta \cdot a_i +1/2\|^2\big\}\right).
\end{equation}
We note that here the exponent is different from that of \eqref{concentration_upper_bound:groups} of the previous section. To handle this difficulty, we decompose $G$ into $G_1=G_1(A):= \{\zeta\in G, \sum_i \|\zeta \cdot a_i\|^2 \ge \sum_i \|\zeta \cdot a_i +1/2\|^2\}$ and $G_2=G_2(A):=\{\zeta\in G, \sum_i \|\zeta \cdot a_i\|^2 < \sum_i \|\zeta \cdot a_i +1/2\|^2\}$. We thus obtain that 
\begin{equation} \label{eqn:fourier4:m}
\rho_m  \le  \frac{1}{|G|} \sum_{\zeta \neq 0, \zeta \in G_1}  \exp(-8\al  \sum_i \|\zeta \cdot a_i +1/2\|^2) + \frac{1}{|G|} \sum_{\zeta \neq 0, \zeta \in G_2}  \exp(-8\al  \sum_i \|\zeta \cdot a_i\|^2):=\Sigma_1 + \Sigma_2.
\end{equation}
Without loss of generality \footnote{The other case will be easier and can be treated similarly as in Section \ref{section:Abelian}.} we assume that 
\begin{equation}\label{eqn:G12}
|\Sigma_1| \ge |\Sigma_2|.
\end{equation}

 We can then proceed similarly as in  the proof of Theorem \ref{theorem:ILO:Abelian} with some modifications. 


\emph{Large level sets.} Let $S_\ell = \{\zeta\in G_1 \big| 4 \alpha \sum_{i=1}^n \|a_i \cdot \zeta +1/2\|^2 \leq \ell\}$. We have
\begin{align*}
	\frac{1}{2m^C} \leq\rho_m(A)/2 \leq \Sigma_2= \frac{1}{|G|}\sum_{\zeta \neq 0, \zeta\in G_2}\exp(-8\alpha\sum_{i=1}^n \big\|a_i \cdot \zeta +1/2 \big\|^2) \leq  \frac{1}{|G|}\sum_{\ell\geq 1} \exp(-2(\ell-1))|S_\ell|.
\end{align*}
Therefore there must be a level set $S_{\ell_0}$ such that 
\begin{align}\label{eqn:Sell_0:m}
	|S_{\ell_0}|\exp(-{\ell_0}+2) \geq \rho_m |G|/2.	
\end{align}

Because $\rho_m \geq 1/m^C|A|$ and of course $|S_{\ell_0}|\leq |G|$ we have $\ell_0 \le C \log m$.

\emph{Double counting and the triangle inequality.} By double counting we have
\begin{align*}
	\sum_{i=1}^n 4\alpha \sum_{\zeta\in S_{\ell_0}}  \| a_i \cdot \zeta +1/2\|^2 = \sum_{\zeta\in S_{\ell_0}} 4\alpha \sum_{i=1}^n  \| a_i \cdot \zeta+1/2\|^2\leq {\ell_0}|S_{\ell_0}|. 	
\end{align*}
So by averaging, at least $n-n'$ of the  $a_i$ satisfy
\begin{align}\label{averaging:groups:m}
	\sum_{\zeta\in S_{\ell_0}} \| a_i \cdot \zeta+1/2\|^2 \leq \frac{{\ell_0}}{ 4\alpha n'}|S_{\ell_0}|.
\end{align}
Let the set of $a_i$ satisfying \eqref{averaging:groups:m} be $A'$. The set $A\setminus A'$ will be our exceptional set. It remains to show that $A'$ is contained in a proper GAP.

Let $j$ be any positive integer. By Cauchy-Schwarz we have for any $a\in (2j)A'$, writing $a=\sum_{i=1}^{2j} a_i$, 
$$\|a \cdot \zeta\|^2=\|a \cdot \zeta+j\|^2   = \|\sum_{i=1}^{2j} a_i \cdot \zeta + j\|^2 =  \|\sum_{i=1}^{2j} (a_i \cdot \zeta + 1/2)\|^2 \le 2j \sum_{i=1}^{2j}  \|(a_i \cdot \zeta + 1/2)\|^2.$$
Furthermore, for any $a\in (2j+1)A'$, writing $a=\sum_{i=1}^{2j+1} a_i$, 
$$\|a \cdot \zeta+1/2\|^2  =\|a \cdot \zeta+j+1/2\|^2= \|\sum_{i=1}^{2j+1} a_i \cdot \zeta + j+1/2\|^2 =  \|\sum_{i=1}^{2j+1} (a_i \cdot \zeta + 1/2)\|^2 \le (2j+1) \sum_{i=1}^{2j}  \|(a_i \cdot \zeta + 1/2)\|^2.$$

Set 
\begin{equation}\label{eqn:Ab:k:m}
k :=\left\lfloor \sqrt{\frac{{\alpha}n' }{100{\ell_0}}} \right\rfloor.
\end{equation}
Note that 
$$k \asymp \sqrt{m n'/100 \ell_0 n}  \ge \sqrt{m/ \eps 100 \ell_0} \ge m^{1/3}$$ 
if we assume that $m$ is sufficiently large, given $\eps$. 

\emph{Dual sets.} Define  
$$S_{\ell_0,1}^* := \Big \{a\ \big|\sum_{\zeta\in S_{\ell_0}} \| a\cdot \zeta\|^2 \leq \frac{1}{200}|S_{\ell_0}| \Big\}$$ 
and 
$$S_{\ell_0,2}^* := \Big \{a\ \big|\sum_{\zeta\in S_{\ell_0}} \| a\cdot \zeta +1/2\|^2 \leq \frac{1}{200}|S_{\ell_0}| \Big\}.$$ 

By \eqref{averaging:groups:m}, by the choice of $k$, and by Cauchy-Schwarz estimates above we have  that
$$\bigcup_{1\le l \le k/2} (2l)A' \subset S_{\ell_0,1}^*$$
and
 $$\bigcup_{1\le l \le (k-1)/2} (2l+1)A' \subset S_{\ell_0,2}^*.$$

Setting $A'' = A'\cup \{0\}$ we have 
$$kA'' = \{0\}  \cup \bigcup_{l=1}^{k} lA' \subset S_{\ell_0,1}^* \cup S_{\ell_0,2}^* .$$
As  in Section \ref{section:Abelian} we have 
$$|S_{\ell_0,1}^*| \le \frac{4|G|}{|S_{\ell_0}|}.$$ 
We claim similarly for $S_{\ell_0,2}^*$ that
\begin{align}\label{bound for dual set:groups:S2}
	|S_{\ell_0,2}^*|\leq \frac{4|G|}{|S_{\ell_0}|}.
\end{align}
Indeed, define $T_a = -\sum_{\zeta\in S_{\ell_0}} \cos (2\pi a\cdot \zeta)=  \sum_{\zeta\in S_{\ell_0}} \cos (2\pi (a\cdot \zeta +1/2))$. Again, as $\cos (2\pi z) \geq 1 - 100\|z\|^2$  for any $z\in \R$, we have for any $a\in S_{\ell_0,2}^*$
\begin{align*}
	T_a\geq \sum_{\zeta\in S_{\ell_0}} \left(1-100\|a \cdot \zeta+1/2\|^2\right)\geq \frac{1}{2}|S_{\ell_0}|.
\end{align*}
We also have the upper bound on $T_a$ given by
\begin{align*}
	\sum_{a\in G} T_a^2 &\leq \sum_{a\in G} \Big(-\sum_{\zeta\in S_{\ell_0}} \cos(2\pi a\cdot \zeta)  \Big)^2 \\
	&\leq \sum_{a\in G} \sum_{\zeta_1,\zeta_2\in S_{\ell_0}} e(2\pi a \cdot (\zeta_1-\zeta_2)) \\
	&= |G| |S_{\ell_0}|.
\end{align*}

Putting these bounds together we have that, with $A''= A' \cup \{0\}$, 
\begin{align*}
	|kA''| \leq |S_{\ell_0,1}^*| + |S_{\ell_0,2}^*| \leq 2 \frac{4|G|}{|S_{\ell_0}|} \leq  8\rho_m^{-1}\exp(-{\ell_0}+2),
\end{align*}

\emph{Long Range Inverse Theorem.} Recall our hypothesis that $\rho_\xi(A)\geq 1/m^{-C}$. Therefore, with $k$ from \eqref{eqn:Ab:k:m} we clearly have 
$$|kA''| \le  \frac{k^{4C}}{2^{2^{2^{28C+C_0}}}} \le  \frac{k^{4C}}{2^{2^{2^{28C+C_0}}}}|A''|$$
 if we assume that $C =o( \log \log \log m)$. 
 
 It thus follows from Theorem~\ref{longrangeinversetheorem:Abelian} that $A''$ is contained in a symmetric proper coset-progression $H+Q$ of rank $4C$ and size 
\begin{align*}
|H+Q| & \le 2^{2^{2^{28C+C_0}}} 8 \exp(-\ell_0+2) \frac{(\rho_m(A))^{-1}}{{k}^{r}} \\
& \le 2^{2^{2^{28C+C_0}}} 8 \exp(-\ell_0+2) (\sqrt{100 \ell_0})^{r}   \frac{(\rho_m(A))^{-1}}{(\al n')^{r/2}}\\
& \le  2^{2^{2^{28C+C_0'}}} \frac{(\rho_m(A))^{-1}}{(m n'/n)^{r/2}},
\end{align*}
concluding the proof of Theorem \ref{theorem:ILO:m:Abelian} in its quantitative form of \eqref{eqn:Ab:m:HP:quant}.

\section{Inverse results with constraints over real numbers: proof of Theorem \ref{theorem:ILO:ast} and \ref{theorem:ILO:ast:sparse}}

It suffices to prove the more general result, Theorem \ref{theorem:ILO:ast:sparse}. Notice that we are in the torsion-free setting. We will first pass to the iid case. Set 
$$\alpha:=\frac{m}{n}.$$
We then see that $n^{\eps -1} \le \al \le 1- n^{\eps -1}$.

Let $\xi$ be a Bernoulli random variable with parameter $\alpha$, i.e. 
$$\P(\xi=1) =\alpha \mbox{ and } \P(\xi=0)=1-\alpha.$$
We recall that  $\rho_\xi(A) = \sup_{a} \P\Big(\sum_{i=1}^n a_ix_i =a\Big)$ where $x_i$ are iid copies of $\xi$. We will pass from $\rho_m^\ast$ to $\rho_\xi$ using the following claim
\begin{claim}\label{claim:passing} We have
$$ \rho_\xi(A) =\Omega(\rho_{m}^\ast(A)/\sqrt{n\alpha}).$$
\end{claim}
\begin{proof} Using Stirling approximation, 
$$\P(x_1+\dots+ x_n =m) \asymp 1/\sqrt{n\alpha}$$ 
and conditioning on $\sum_i x_i=m$ 
$$\P\left(x_{i_1}=1,\dots, x_{i_m}=1, x_{j}=0,j\notin\{i_1,\dots,i_m\}=0 | \sum_{i=1}^n x_i=m\right)= \frac{1}{\binom{n}{m}}.$$
Hence
$$ \sup_{a} \P\Big(\sum_{i=1}^n a_ix_i =a  | \sum_{i=1}^n x_i =m \Big)=\rho_m^\ast(A).$$
So
\begin{align*}
\rho_\xi(A) &=  \sup_{a} \P\Big(\sum_{i=1}^n a_ix_i =a\Big)\\
&\ge  \sup_{a} \P\Big(\sum_{i=1}^n a_ix_i =a  \wedge \sum_{i=1}^n x_i =m \Big) \\
&=  \sup_{a} \P\Big(\sum_{i=1}^n a_ix_i =a  \big| \sum_{i=1}^n x_i =m \Big) \P\big( \sum_{i=1}^n x_i =m\big) \\
&\ge \rho_m^\ast(A)/\sqrt{n\alpha}. 
\end{align*}
\end{proof}
By Claim \ref{claim:passing}, it suffices to study $\rho_\xi(A)$. 

\subsection{Proof of Theorem \ref{theorem:ILO:ast:sparse}}
We will again follow the approach of \cite{NgV-Adv} (as well as the proofs of Theorem \ref{theorem:ILO:Abelian}, Theorem \ref{theorem:ILO:m:Abelian}, and of \cite[Theorem 7.3]{NgW}) with some major modifications.

First of all, by using Freiman-isomorphism (see for instance \cite{NgV-Adv}) it suffices to assume $A$ to be a set of integers. In what follows we will choose $N$ to be a sufficiently large integer (given $A$), and $p$ to be a sufficiently large prime number given $N$ (such as $p\geq 2^n\sum_{i=1}^n (|a_i|+N+1)$).  We will work over modulo $p$. 

We consider the translation $A'=\{a_1',\dots,a_n'\}$ of $A$, where $a_i':=a_i+N$.
In what follows, we compute $\rho_\xi(A')$ using discrete Fourier analysis. Writing $e_p(x) = \exp(2\pi \sqrt{-1}x/p)$, we have the standard identity
\begin{align*}
	\rho_\xi(A') = \P(S = a) = \E\frac{1}{p}\sum_{\zeta\in\F_p} e_p(\zeta(S - a)) = \E\frac{1}{p}\sum_{\zeta\in\F_p} e_p(\zeta S)e_p(-\zeta a),
\end{align*}
with $S=\sum_{i=1}^n a_i'x_i$ and $x_i$ are iid copies of $\xi$ (i.e. Bernoulli random variables taking values $0$ with probability $1-\alpha$ and $1$ otherwise). Let $y_i$ be iid symmetrized versions of $x_i$. In other words $y_i = x_i - x_i'$ with $x_i'$ another iid copy of $x_i$. Then let $y_i'$ be a lazy version of $y_i$. Explicitly,
	$$\P(y_i' = z) =
	\begin{cases}
		\alpha(1-\alpha)/2 & \text{ if } z=\pm 1 \\ 
		\frac{1}{2} + \frac{1}{2}((1-\alpha)^2 + \alpha^2) & \text{ if } z=0 
	\end{cases}.
	$$
	Define $\alpha' = \alpha(1-\alpha)/2$. In other words $\P(y_i' = \pm 1) = \alpha'$ and $\P(y_i' = 0) = 1- 2\alpha'$. Observe $\alpha'\leq \min(1/8, \alpha/2)$.

By independence 
\begin{align*}
	\E e_p( \zeta S) = \prod_{i=1}^n \E e_p(\zeta x_i a_i') \leq \prod_{i=1}^n (\frac{1}{2}(|\E e_p(\zeta x_i a_i')|^2 + 1)) = \prod_{i=1}^n |\E e_p(\zeta y_i' a_i')| = \prod_{i=1}^n (1-2\alpha' + 2\alpha'\cos(2\pi\zeta a_i'/p)).
\end{align*}
Hence
\begin{align*}
|\rho_\xi(A') -1/p| = |\P(S = a)-1/p| &\le \frac{1}{p}\sum_{\zeta\in\F_p, \zeta \neq 0}  \prod_{i=1}^n |(1-2\alpha' + 2\alpha'\cos(2\pi\zeta a_i'/p))|\\
&=  \frac{1}{p}\sum_{\zeta\in\F_p, \zeta \neq 0}  \prod_{i=1}^n |(1-2\alpha' + 2\alpha'\cos(\pi\zeta a_i'/p))|
\end{align*}
where we dropped the $2$ as multiplication by $2$ is a bijection on $\F_p$ for $p>2$.

Again by using $|\sin \pi z| \geq 2\|z\|$ and $|\cos \frac{\pi x}{p}| \leq 1 - \frac{1}{2}\sin^2 \frac{\pi x}{p} \leq 1 - 2\big\|\frac{x}{p}\big\|^2$ we have
\begin{align*}
	0\le 1 - 2\alpha' + 2\alpha'\cos(\pi\zeta a_i'/p) \leq 1-2\alpha' + 2\alpha'(1 - 2\|\zeta a_i'/p\|^2) \leq 1 - 4\alpha'\|a_i'\zeta /p\|^2 \leq \exp(-4\alpha'\|\frac{a_i'\zeta}{p}\|^2).
\end{align*}

Therefore we obtain the following inequality,
\begin{align}\label{concentration_upper_bound}
	\rho_\xi(A') \leq  \frac{1}{p}\sum_{\zeta\in \F_p}\exp(-4\alpha' \sum_{i=1}^n \big\|\frac{a_i'\zeta}{p}\big\|^2).
\end{align}

Combining with our assumption on $\rho_\xi(A')$ yields
\begin{align*}
	n^{-C} \leq \frac{1}{p}\sum_{\zeta\in \F_p}\exp(-4\alpha'\sum_{i=1}^n \big\|\frac{a_i'\zeta}{p}\big\|^2).
\end{align*}

\emph{Large level sets.} Now we split up the $\zeta$ by their effect on the sum. Let $S_\ell = \{\zeta\ \big| 2\alpha' \sum_{i=1}^n \|a_i'\zeta/p\|^2 \leq \ell\}$. We have
\begin{align*}
	n^{-C}\leq\rho_\xi(A')\leq\frac{1}{p}\sum_{\zeta\in \F_p}\exp(-4\alpha'\sum_{i=1}^n \big\|\frac{a_i'\zeta}{p}\big\|^2) \leq  \frac{1}{p}\sum_{\ell\geq 1} \exp(-2(\ell-1))|S_\ell|.
\end{align*}

Because $\rho_\xi\geq n^{-C}$ and of course $|S_{\ell_0}|\leq p$ we have $\ell_0 = O(\log n)$.

\emph{Double counting and the triangle inequality.} By double counting we have
\begin{align*}
	\sum_{i=1}^n2\alpha'\sum_{\zeta\in S_{\ell_0}}  \|\frac{a_i'\zeta}{p}\|^2 = \sum_{\zeta\in S_{\ell_0}}2\alpha'\sum_{i=1}^n  \|\frac{a_i'\zeta}{p}\|^2\leq {\ell_0}|S_{\ell_0}|. 	
\end{align*}
So by averaging, at least $n-n'$ of the  $a_i'$ satisfy
\begin{align}\label{averaging}
	\sum_{\zeta\in S_{\ell_0}} \|\frac{a_i'\zeta}{p}\|^2 \leq \frac{{\ell_0}}{2\alpha'n'}|S_{\ell_0}|.
\end{align}
Let the set of $a_i'$ satisfying \eqref{averaging} be $A_1'$. The set $A'\setminus A_1'$ will be our exceptional set. It remains to show that $A_1'$ is contained in a proper GAP. 

Note that as before, by Cauchy-Schwarz, for any $a \in jA_1'$ where $j\leq k$, we also have
\begin{align}\label{longrangeinverse}
	\sum_{\zeta\in S_{\ell_0}}\Big\|\frac{a\zeta}{p}\Big\|^2\leq k^2 \frac{{\ell_0}}{2\alpha'n'}|S_{\ell_0}|.
\end{align}
\emph{Dual sets.} Define 
$$S_{\ell_0}^* := \Big \{a\ \big|\sum_{\zeta\in S_{\ell_0}} \|a\zeta/p\|^2 \leq \frac{1}{200}|S_{\ell_0}| \Big\}.$$ 
As shown in the previous sections, we also have
\begin{align}\label{bound for dual set}
	|S_{\ell_0}^*|\leq \frac{4p}{|S_{\ell_0}|}.
\end{align}

Set 
$$k :=\left\lfloor \sqrt{\frac{\alpha' n' }{200{\ell_0}}} \right\rfloor.$$
Then clearly $k\ge n^{\eps/2}$ if $n$ is sufficiently large.

By \eqref{longrangeinverse} we have  that
$$\bigcup_{l=1}^{2k} lA_1' \subset S_{\ell_0}^*.$$ 
Setting $A_1'' = A_1'\cup \{0\}$ we have 
$$2kA_1'' = \{0\}  \cup \bigcup_{l=1}^{2k} lA_1' \subset S_{\ell_0}^*.$$
This gives us the bound
\begin{align*}
	|2kA_1''| \leq |S_{\ell_0}^*| \leq \frac{4p}{|S_{\ell_0}|} \leq  4\rho_\xi^{-1}\exp(-{\ell_0}+2),
\end{align*}
where in the second inequality we used \eqref{bound for dual set}.

From now on our treatment is different from the previous sections. Let $A_1$ be the set of elements $a$ of $A$ for which $a+N \in A_1'$. Pick $a_1\in A_1$, i.e. $a_1+N \in A_1'$ and use it to define
$$B_1:=\{a-a_1, a\in A_1\}.$$ 
Then we can write 
$$lA_1' = lA_1 + lN = l(a_1+N) + lB_1.$$
Now as $B_1$ contains 0, we have (with the convention that $ 0B_1=\{0\}$)
$$lB_1 = \bigcup_{l'=0}^{l} l'B_1.$$  
So together we have 
\begin{align*}
    2k A_1'' &=  \bigcup_{l=0}^{2k}  \bigcup_{l'=0}^{l} l'B_1 + l(a_1+N) =  \bigcup_{l'=0}^{2k}  \bigcup_{l=l'}^{2k} l'B_1 + l(a_1+N) \supseteq  \bigcup_{l'=k}^{2k}  \bigcup_{l=k}^{2k} l'B_1 + l(a_1+N),
\end{align*}
and it thus follows that 
	$$ \Big|\bigcup_{l'=k}^{2k} l'B_1 + \{k,\dots, 2k\}(a_1 + N)\Big| \le |2kA_1''| \le  4\rho_\xi^{-1}(A')\exp(-\ell_0 + 2) = O\big((\rho_m^\ast)^{-1}(A) \sqrt{\alpha n}  \exp(-\ell_0+2)\big).$$
Choosing $N$ sufficiently large, the sets $\bigcup_{l'=0}^{l} l'B_1 + k'(N+a_1)$ where $k\leq l'\leq 2k$ are disjoint. We thus obtain that 
\begin{align*}
    \Big|\bigcup_{l'=0}^{l} l'B_1 \Big|
    &= O\Big((\rho^\ast_m)^{-1}(A) \sqrt{\alpha n}  \exp(-\ell_0+2)/k\Big)\\
&= O\left((\rho^\ast_m)^{-1}(A) \sqrt{\alpha n}  \exp(-\ell_0+2)\sqrt{\frac{200\ell_0}{\alpha'n'}}\right)\\
&=O\left( \sqrt{\frac{n}{n'}}(\rho^\ast_m)^{-1}(A)  \exp(-\ell_0+2)\sqrt{200\ell_0} \right).
\end{align*}
Note that $1/4\leq \alpha'/\alpha \leq 1/2$ so the ratio may be absorbed into the constant. Again because $0\in B_1$,
\begin{align}\label{lri_condition}
|kB_1 |= O\left(\sqrt{\frac{n}{n'}} (\rho^\ast_m)^{-1}(A) \exp(-\ell_0+2)\sqrt{200\ell_0} \right).
\end{align}

\emph{Long Range Inverse Theorem.} We now consider our $v_i$ as integers again. Remember we chose $p$ so large that $kB_1$ should be the same size in the integers as in the integers mod $p$. Recall our hypothesis that $\rho^*_m(A)\geq n^{-C}$. Therefore $\rho^{-1}_m\leq k^{2C/\eps}$ and it follows from Theorem~\ref{longrangeinversetheorem} that $B_1$ is contained in a symmetric proper GAP $P$ of rank $r \le 2C/\eps$ and size 
\begin{align*}
|P| & =O\left( 2^{2^{2^{14C+C_0}}} \sqrt{\frac{n}{n'}}  (\rho^\ast_m)^{-1}(A)  \exp(-\ell_0+2)\sqrt{200\ell_0}/{k}^{r'}\right) \\
& = O\left(  2^{2^{2^{28C+C_0}}} \sqrt{\frac{n}{n'}}  (\rho^\ast_m)^{-1}(A)/(\alpha'n')^{r'/2}\right)  \\
& = O\left(  2^{2^{2^{28C+C_0}}} \sqrt{\frac{n}{n'}}  (\rho^\ast_m)^{-1}(A)/(\alpha n')^{r'/2}\right) 
\end{align*}
such that $B_1 \subset P$. It thus follows that $A_1 \subset a_1+P$, concluding the proof of Theorem \ref{theorem:ILO:ast:sparse} in its quantitative form \eqref{eqn:ILO:ast:sparse:quant}.


\section{Proof of the corollaries}\label{appendix:lemma}


 \begin{proof}(of Corollary \ref{cor:Abelian:1}) Assume otherwise that there exists $a$ such that $\P\Big(\sum_{i=1}^n a_ix_i =a\Big) \ge K_\eps/\sqrt{n}$ for sufficiently large $K_\eps$. (For instance $K_\epsilon > C_\epsilon + C_\epsilon^{-1}$ would work.) Then 
 $$\rho_\xi(A) \ge  K_\eps/\sqrt{n} - 1/|G| \ge (K_\eps-C_\epsilon^{-1})/\sqrt{n}.$$
 Applying (ii) of Theorem \ref{theorem:ILO:Abelian} with $C=1/2$ and $n'=\lfloor \eps n \rfloor$ we obtain a symmetric proper coset-progression $H+P$ with rank $r\ge 0$ and size at most $\max\{1, C_\eps n^{1/2}/ (K_{\eps} -C_\epsilon^{-1}) (\eps n)^{r/2}) \}$ which contains at least $n-n'$ elements of $\{2a, a\in A\}$. There are two cases to consider: (1) If $r \ge 1$ then $\max\{1, C_\eps n^{1/2}/ (K_{\eps} -C_\epsilon^{-1}) (\eps n)^{r/2}) \}=1$. This implies $|H+P|=1$ or equivalently $H+P = \{0\}$, a contradiction because $\{2a, a\in A\}$ cannot have more than $n-n'$ zero elements. (2) If $r=0$ then $H+P=H$, which would be a subgroup of size at most $C_\eps n^{1/2}$ that contains  at least $n-n'$ elements of $\{2a, a\in A\}$, contradicting our assumption. 

For the consequence, view $\Z/q\Z$ as $\{0,\dots, q-1\}$ and assume that $H$ consists of elements of form $rb$, where $r|q$. As this group is of size at most $C_\eps \sqrt{n}$, we have $r>2$. Now if $H$ contains at least $n-n'$ elements of $\{2a,a\in A\}$, then $H$ must contain at least one element $2a$, where $a\in A$ is reduced. However this means that $2a =rb$ for some $b$, so either $a=(r/2) b$ or $a=r (b/2)$. In either case $(a,q)>1$, and hence $a$ is not reduced, a contradiction.
\end{proof}

\begin{proof}(of Corollary \ref{cor:forward}) 
For the first statement, assume otherwise that $\rho^\ast(A) \geq K_\eps \sqrt{n}/n'$ for some large constant $K_\eps$. Then we can apply Theorem \ref{theorem:ILO:ast:sparse} to obtain a GAP $P$ containing at least $n-n'$ elements of $A$ such that
$$|P|=\max\left\{1,C_\ep \sqrt{\frac{n}{n'}}n'\big/\left(K_\eps\sqrt{n}(n')^{r'/2}\right)\right\}.$$
Because there are no more than $n-n'-1$ elements of $A$ taking the same value, $P$ must contain at least 2 distinct elements, and hence its rank $r'$ is at least 1. This is a contradiction because, 
$$C_\ep \sqrt{\frac{n}{n'}}n'\big/\left(K_\eps\sqrt{n}(n')^{1/2}\right) = C_\ep/K_\eps < 1, $$
provided that $K_\eps$ is large. 

Likewise, for the second statement, assume otherwise that $\rho^\ast(A) \geq K_\eps n^{-3/2}$ for some large constant $K_\eps$. Apply Theorem \ref{theorem:ILO:ast} for $n'=n/2$ to obtain a GAP $P$ containing at least $n/2$ elements of $A$ and
$$|P|=C_\eps\sqrt{\frac{n}{n/2}}n^{3/2}\big/\left(K_\eps(n/2)^{r'/2}\right) = C_\eps/K_\eps,$$
because $r' \geq 1$. If we choose $K_\eps$ to be large then $|P| < n/2$, which contradicts the fact that $P$ contains at least $n/2$ elements of $A$ (which are all distinct by assumption).
\end{proof}

\begin{proof}(of Corollary \ref{cor:forward:rho_m}) Assume otherwise that $\rho_m^\ast(A) \geq K_\eps \frac{n}{n'\sqrt{m}}$ for some large constant $K_\eps$. Then we can apply Theorem \ref{theorem:ILO:ast} to obtain a GAP $P$ containing at least $n-n'$ elements of $A$, where $$|P|=C_\eps \left(\sqrt{\frac{n}{n'}}K_\eps \frac{n}{n'\sqrt{m}}\right)^{-1} \big/(m n'/n)^{r/2}.$$
Because there are no more than $n-n'-1$ elements of $A$ taking the same value, $P$ must contain at least 2 distinct elements, and hence its rank $r$ is at least 1. This is a contradiction because, 
$$|P|= C_\eps\sqrt{\frac{n}{n'}}\left(K_\eps \frac{n}{n'\sqrt{m}}\right)^{-1} \big/(m n'/n)^{1/2}  = C_\eps/K_\eps < 1, $$
provided that $K_\eps$ is large.
\end{proof}

\begin{proof}(of Corollary \ref{cor:mix}) For the first assertion, assume otherwise that $\rho_m \ge \frac{\delta}{ |G|}$. Choose $n' =\eps n$ and apply Theorem \ref{theorem:ILO:m:Abelian}, then there exists a symmetric proper coset-progression $H+Q$ of rank $r$ which contains all of the $a_1,\dots, a_k$ and of size at most $\max\{1, O_\eps (\delta^{-1} |G| /(m \eps)^{r/2})\}$. By assumption, as $C$ is chosen sufficiently large, we must have $r\ge k$, but then this would yield that $H+Q$ has size 1 because $(m \eps)^{k/2} \ge (C\eps)^{k/2} \delta^{-1} |G|$, a contradiction.

Now for the second assertion we will show that with high probability $a_1,\dots, a_k$ satisfy the above condition, that there is no symmetric coset-progression $H+Q$ of rank $r\le k-1$ containing them. When this is shown, we will have $|\P(X_1+\dots+X_m=a) -\frac{1}{q}| \le \frac{\delta}{q}$ for any $a$, which would clearly imply the $\delta$-mixing requirement.

In what follows we count the number of tuples $(a_1,\dots, a_k)$ that are elements of some ``economical'' $H+Q$, ranging over all symmetric coset-progressions $H+Q$. We first gather information about the potential symmetric coset-progressions $H+Q$ of small size that might contain all $a_i$. Because the $a_i$ are reduced, the rank $r$ of $H+Q$ must be at least 1 (as otherwise all $a_i$ are in $H$, hence not reduced). In $\Z/q\Z$, there are at most $O(\log q)$ ways to choose $H=H_r$ as the subgroup of form $\{a \in \Z/q\Z, r|a\}$.  Now assume that $g_1,\dots, g_r$ have been chosen; for a symmetric coset-progression $H+Q= \{h+ x_{1} g_1+\dots+ x_{r} g_r, |x_j| \le N_j, h\in H, 1\le j \le r\} $ over the generators $g_1,\dots, g_n$, if $a_1,\dots, a_k \in H+Q$, then  
$$a_i = h_i + x_{i1} g_1+\dots+ x_{ir} g_r =: h_i + \Bx_i \cdot \Bg.$$ 
For $1\le j\le r$, let 
$$\al_j =\max\{|x_{1j}|,\dots, |x_{kj}|\}$$ 
and define 
$$H+Q_e := \{h+ x_{1} g_1+\dots+ x_{r} g_r, |x_j| \le \al_i,h\in H, 1\le j \le r\}.$$
Observe that $a_i \in H+Q_e \subset H+Q$ for all $i$. Hence, to avoid double counting in situations such as $a_i \in H+Q \subset H+Q'$, we will assume $H+Q$ to  have the economical form $H+Q_e$ as above. We next show the following counting claim.
\begin{claim}\label{claim:vectorscounting} Let $k,r,s$ be given positive integers where $s$ is sufficiently large. Let $N$ be the number of integral vectors $\Bx_1=(x_{11},\dots, x_{1r}),\dots, \Bx_k=(x_{k1},\dots, x_{kr})$ satisfying $\prod_{j=1}^r \al_j  \le s$, where $\al_j =\max\{|x_{1j}|,\dots, |x_{kj}|\}$, is bounded by 
$$N \le 2^{O(kr)} s^k  (\log s)^{r-1}.$$ 
\end{claim}
Note that this bound for $r=1$ is rather trivial (but near optimal).
\begin{proof}(of Claim \ref{claim:vectorscounting})
	With a loss of a factor of $2^{kr}$ we assume that $x_{ij}$ are positive integers. By a volume packing argument, it suffices to compute the volume of the set $S$ of positive real vectors $\Bx_1,\dots, \Bx_k$ in $[1,\infty)^r$ satisfying $\prod_{j=1}^r \al_j  \le s$.

Viewing $\Bx_1,\dots, \Bx_k$ as vectors chosen uniformly from the box $[1,s]^r$, then $\al_1,\dots, \al_k$ are iid, of distribution function $\P(\al \le x) = (x/(s-1))^k$ for $1\le x\le s$. Hence the joint density of $(\al_1,\dots, \al_r)$ is $k^r (1/(s-1))^r (x_1/(s-1))^{k-1} \dots (x_r/(s-1))^{k-1}$. So
\begin{align*}
	\frac{\Vol(S)}{(s-1)^{rk}} = \P\Big(\prod \al_i \le s\Big) &=\int_{(x_1,\dots, x_r) \in [1,s]^r} 1_{x_1\dots x_r \le s} k^r \Big(\frac{1}{s-1}\Big)^r \Big(\frac{x_1}{s-1}\Big)^{k-1} \dots \Big(\frac{x_r}{s-1}\Big)^{k-1} dx_1\dots dx_r\\
	&= k^r \Big(\frac{1}{s-1}\Big)^{r k}  \int_{[1,s]^{r-1}} (x_1 \dots x_{r-1}\Big)^{k-1} \bigg[\int_1^{s/x_1\dots x_{r-1}} x_r^{k-1}dx_r\bigg] dx_1\dots dx_{r-1}\\
	&=k^{r-1}   \Big(\frac{1}{s-1}\Big)^{r k} s^k  \int_{[1,s]^{r-1}} (1/x_1) \dots (1/x_{r-1}) dx_1\dots dx_{r-1} \\
	& = k^{r-1}   \Big(\frac{1}{s-1}\Big)^{r k} s^k (\log s)^{r-1}.
\end{align*}
It thus follows that $\Vol(S) \le k^r s^k  (\log s)^{r-1}$, and hence the number of vectors $\Bx_1,\dots, \Bx_k$ of (not necessarily positive) integral entries satisfying  $\prod_{j=1}^r \al_j  \le s$ is bounded by $ 2^{O(kr)} s^k  (\log s)^{r-1}$ as desired.
\end{proof}

Using the above claim,  for a given $H$ of size $h$, for each $1\le r \le k-1$, there are at most $q^r$ ways to choose the generators $g_1,\dots, g_r$ and form a $Q_e$ of volume at most $s=C_\eps q /m^{r/2} h$ using these generators. By the above claim the set $\CC_{r,H}$ of such vectors $a_1,\dots, a_k$ (i.e. $(h_1,\Bx_1),\dots, (h_k,\Bx_k)$) that are contained in an economical  symmetric coset-progressions $H+Q_e$ of size at most $h \times (C_\eps \delta^{-1} q /h m^{r/2})$ is bounded by 
$$O\Big(q^r 2^{kr} h^k (C_\eps  \delta^{-1} q /m^{r/2} h)^k \log^{r-1}q\Big)$$
where the first factor $q^r$ comes from the choices of $g_1,\dots, g_r$, the second factor $2^{kr}$ comes from the signs of $x_{ij}$, the third factor $h^k$ comes from the choices of $h_i\in H$, and the remaining factors come from Claim \ref{claim:vectorscounting}.
 
The probability that a random reduced tuple $(a_1,\dots, a_k)$ is one of such tuple (contained by structure) is thus bounded by
\begin{align*}
\frac{1}{(\phi(q))^k}\sum_{H}\sum_{r=1}^{k-1}  O\Big (q^r  2^{kr} h^k (C_\eps  \delta^{-1} q /m^{r/2} h)^k \log^{r-1}q \Big)   & = \Big(\frac{c}{\log \log q}\Big)^k O\Big(\sum_{r=1}^{k-1} 2^{kr} (C_\eps^k \delta^{-k} q^r/m^{rk/2})  \log^{r}q \Big) \\
& = \Big(\frac{C'_\eps}{\log \log q}\Big)^k O\Big(\sum_{r=1}^{k-1} (q (\log q)/ (\delta^2 m)^{k/2})^r\Big)\\
&= O(t^{-k/2})
\end{align*}
where we used the fact that $m \ge C_{\eps}'' t \delta^{-2} (\log \log q)^2 (q \log q)^{2/k}$ and that $\phi(q) \ge c q/\log \log q$ for a sufficiently small constant $c$.


\end{proof}

{\bf Acknowledgements.} We would like to thank V. Vu for helpful discussion and suggestions. 

\end{document}